\documentclass[11pt,a4paper,reqno]{amsart}
\usepackage{amsfonts}
\usepackage{amsmath,amssymb,amsthm,amsxtra}
\usepackage{float}
\usepackage[colorlinks, linkcolor=blue,anchorcolor=Periwinkle,
    citecolor=red,urlcolor=Emerald]{hyperref}
\usepackage[usenames,dvipsnames]{xcolor}
\usepackage{geometry,array} 
\usepackage{graphicx,subfigure,bookmark,tikz,pgfplots,url}
\title{Global dimension function on stability conditions and Gepner equations}
\author{Yu Qiu}
\address{Qy:
	Yau Mathematical Sciences Center and Department of Mathematical Sciences,
	Tsinghua University,
    100084 Beijing,
    China.
    \&
    Beijing Institute of Mathematical Sciences and Applications, Yanqi Lake, Beijing, China}
\email{yu.qiu@bath.edu}

\usetikzlibrary{matrix,positioning,decorations.markings,arrows,decorations.pathmorphing,
    backgrounds,fit,positioning,shapes.symbols,chains,shadings,fadings,calc}

\tikzset{->-/.style={decoration={  markings,  mark=at position #1 with
    {\arrow{>}}},postaction={decorate}}}
\tikzset{-<-/.style={decoration={  markings,  mark=at position #1 with
    {\arrow{<}}},postaction={decorate}}}

\usepackage[all]{xy}

\usepackage{mathrsfs}
\usepackage[frak=esstix]{mathalfa}
\def\XX{\mathbb{X}}

\def\QStab{\operatorname{QStab}}

\theoremstyle{plain}
\newtheorem{theorem}{Theorem}[section]

\newtheorem{lemma}[theorem]{Lemma}
\newtheorem{corollary}[theorem]{Corollary}
\newtheorem{proposition}[theorem]{Proposition}
\newtheorem{conjecture}[theorem]{Conjecture}
\theoremstyle{definition}
\newtheorem{definition}[theorem]{Definition}

\newtheorem{example}[theorem]{Example}

\newtheorem{remark}[theorem]{Remark}

\numberwithin{equation}{section}

\def\hh{\mathcal}

\def\kong{\mathbb}
\def\<{\langle}
\def\>{\rangle}

\def\ZZ{\mathbb{Z}}

\def\R{\mathbb{R}}
\def\bP{\mathbb{P}}

\def\CC{\mathbb{C}}

\def\Aut{\operatorname{Aut}}
\def\Ind{\operatorname{Ind}}
\def\Sim{\operatorname{Sim}}
\def\Hom{\operatorname{Hom}}

\def\Stab{\operatorname{Stab}}

\def\diff{\operatorname{d}}
\def\Br{\operatorname{Br}}

\def\deg{\operatorname{deg}}
\newcommand{\h}{\hh{H}}            

\renewcommand{\k}{\mathbf{k}}
\renewcommand{\mod}{\operatorname{mod}}

\newcommand{\Cone}{\operatorname{Cone}}

\renewcommand{\Re}{\operatorname{Re}}

\newcommand{\id}{\operatorname{id}}

\newcommand{\D}{\operatorname{\hh{D}}}

\newcommand{\per}{\operatorname{per}}

\def\arrow{red}

\newcommand{\ST}{\operatorname{ST}}        

\def\coh{\operatorname{coh}}

\def\RHom{\operatorname{RHom}}
\def\ee{\operatorname{\mathcal{E}}}

\newcommand{\twi}{\Psi} 

\def\gldim{\operatorname{gldim}}
\def\XStab{\QStab}

\def\sli{\mathcal{P}}
\newcommand{\norm}[1]{\lVert #1 \rVert}
\def\hua{\mathcal}
\def\HMF{\operatorname{HMF}}
\def\RA{\overline{\tau}}
\def\AR{\tau}
\def\DQ{\D_\infty(Q)}
\def\DXQ{\D_\XX(Q)}
\newtheorem{question}{Question}[section]
\newcommand{\qq}[1]{\operatorname{\Gamma}_{#1}Q}
\def\cd{\operatorname{gd}_Q}
\def\Lag{\hua{L}}

\def\ssl{\sigma_{G,s}^{\Lag}}
\def\txl{\AR_\XX^{\Lag}}
\def\zql{\zeta_Q^{\Lag}}
\def\Poly{\operatorname{Poly}}

\begin{document}

\begin{abstract}
    We study the global dimension function $\gldim\colon\Aut\backslash\Stab\D/\CC\to\R_{\ge0}$
    on a quotient of the space of Bridgeland stability conditions
    on a triangulated category $\D$
    as well as Toda's Gepner equation $\Phi(\sigma)=s\cdot\sigma$
    for some $\sigma\in\Stab\D$ and $(\Phi,s)\in\Aut\D\times\CC$.

    For the bounded derived category $\D^b(\k Q)$ of a Dynkin quiver $Q$,
    we show that there is a unique minimal point $\sigma_G$ of $\gldim$ (up to the $\CC$-action),
    with value $1-2/h$,
    which is the solution of the Gepner equation $\AR(\sigma)=(-2/h)\cdot\sigma$.
    Here $\AR$ is the Auslander-Reiten functor and $h$ is the Coxeter number.
    This solution $\sigma_G$ was constructed by Kajiura-Saito-Takahashi.
    We also show that for an acyclic non-Dynkin quiver $Q$,
    the minimal value of $\gldim$ is $1$.

    Our philosophy is that the infimum of $\gldim$ on $\Stab\D$
    is the global dimension for the triangulated category $\D$.
    We explain how this notion could shed light on the contractibility conjecture
    of the space of stability conditions.
    

    \vskip .3cm
    {\parindent =0pt
    \it Key words:}
    global dimension, stability conditions, Gepner equations

\end{abstract}
\maketitle

\section{Introduction}

\subsection{Global dimension for triangulated categories}
A Bridgeland stability condition $\sigma=(Z,\hh{P})$ on a triangulated category $\D$
consists of a central charge $Z\in\Hom(K(\D),\CC)$ and
a slicing $\hh{P}=\{\hh{P}(\phi)\mid\phi\in\R\}$, which is a collection of additive/abelian subcategories.
All stability conditions on $\D$ form a complex manifold $\Stab\D$ (\cite{B1}),
which is a homological invariant of $\D$.
The study of spaces of stability conditions already has many applications in several topics,
such Donaldson-Thomas theory, cluster theory and (homological) mirror symmetry
(cf. e.g. \cite{BS,DHKK,KoSo}).

With the following two motivations:
\begin{enumerate}
\item to understand the conjectural Frobenius-Saito structure on the space of stability conditions,
cf. \cite{I,BQS},
\item to understand the theory that realizes stability conditions on Fukaya type categories
via quadratic differentials on Riemann surfaces, cf. \cite{BS}.
\end{enumerate}
We have introduced $q$-deformations of stability conditions in \cite{IQ1}.
There, we have also introduced a function, the \emph{global dimension function},
\[\gldim\hh{P}=\sup\{ \phi_2-\phi_1 \mid
    \Hom(\hh{P}(\phi_1),\hh{P}(\phi_2))\neq0\}\]
on the set of (slicings of)  stability conditions $\sigma=(Z,\hh{P})$.
This function plays a key role in the existence of $q$-deformations of stability conditions.

In the case where $\D=\D^b(A)$ for some algebra $A$ and
where its canonical heart $\h=\mod A$
is concentrated in one slice $\hh{P}(\phi_0)$
(which implies that $\hh{P}(\phi)$ vanishes if $\phi-\phi_0\notin\ZZ$),
we have
\[
    \gldim\hh{P}=\gldim \h =\gldim A,
\]
where $\gldim A$ is the classical global dimension of the algebra $A$.
Our philosophy is that
the infimum of $\gldim$ on $\Stab\D$ is the categorical global dimension
for a triangulated category $\D$, i.e.
\[\boxed{\gldim\D\colon=\inf \gldim(\Stab\D).}\]
For instance, when the category $\D$ is the bounded derived category
$$\DQ\colon=\D^b(Q)$$ of a Dynkin quiver $Q$,
we have $$\gldim\DQ=1-2/h,$$ where $h$ is the Coxeter number.
This refines the usual global dimension of the path algebra of a Dynkin quiver.

\subsection{Gepner equations}
A {\em Gepner equation} on the space $\Stab\D$ is an equation of the form
\[\Phi(\sigma)=s\cdot\sigma,\]
where the parameters $(\Phi,s)\in\Aut\D\times\CC$ consist
of an autoequivalence and a complex number.

In the setting of \cite{IQ1,IQ2}, the category
$\D=\D_\XX$ admits a distinguished autoequivalence $\XX$
which yields an isomorphism between the Grothendieck group $K(\D_\XX)$
and $(\ZZ[q,q^{-1}])^{\oplus n}$.
Then the Gepner equation $\XX(\sigma)=s\cdot\sigma$ defines
what we call $q$-stability conditions on $\D_\XX$.

In the usual setting, where one has $K(\D)\cong\ZZ^n$, Gepner
equations where studied by Toda \cite{T1,T2,T3}, who was interested
in finding stability conditions with a symmetry property.
In this case, a solution to such an equation is called a Gepner point.
It is an orbifold point in $\Aut\backslash\Stab\D/\CC$.
Toda's motivation came from the study of Donaldson--Thomas invariants, in particular
for B-branes on Landau-Ginzburg models associated to a superpotential.
Toda constructed the central charge of a conjectural Gepner point $\sigma_G$
for categories $\D=\HMF(f)$ that arise as homotopy categories of graded matrix factorizations.

In the Dyknin case, i.e. when $\HMF(f)$ is derived equivalent to $\DQ$
for a Dynkin quiver $Q$,
Kajiura-Saito-Takahashi \cite{KST} constructed such a Gepner point $\sigma_G$,
a solution of equation \eqref{eq:RA=}, in the sense of Toda
(actually before Toda's work).
The non trivial but very interesting fact is that $\sigma_G$ is
actually the only solution of the Gepner equation $\tau(\sigma)=(-2/h)\cdot\sigma$,
where $\AR$ is the Auslander-Reiten functor and $h$ is the Coxeter number.

This stability condition $\sigma_G$ is the most stable one,
in the sense that
\begin{itemize}
  \item $\sigma_G$ is {\em totally stable}, i.e.
  every indecomposable object is stable.
  Moreover, it is the only stability condition, up to the $\CC$-action,
  where the function $\gldim$ takes its minimal value $1-2/h$.
  \item The phases between stable objects are spread out evenly, cf. Figure~\ref{fig:7},
  where the diagonals of a regular $(n+1)$-gon give the central charges of
  all the indecomposable objects in the $A_n$-case.
\end{itemize}
The notion of total (semi)stability is related to a conjecture of Reineke \cite{R},
which asserts the module category of any Dynkin quiver admits a totally (semi)stable stability condition (see \cite[\S~7.4]{Q2} for a partial answer).
We show that a stability condition $\sigma$ is totally semistable if and only if
$\gldim \sigma \le1$ (Proposition~\ref{pp:leq}).

It would be interesting to calculate all possible Gepner points, say in the case
where $\D=\DQ$.
For instance, when the parameter $s$ in $(\Phi,s)$ is zero,
then Gepner points correspond to stability conditions for certain species
which are obtained from $Q$ by folding (cf. \cite{Q3}).

\subsection{Contents}
The paper is organized as follows.

In Section~\ref{sec:pre}, we collect basic facts about stability conditions.
In Section~\ref{sec:gldim}, we introduce the global dimension function $\gldim$
and study its general properties. In particular,
we discuss its relation with total stability and Serre functors.
For a type A quiver, we give a description of all total stability conditions in
Proposition~\ref{pp:A} and an explicit formula for calculating $\gldim$ in such
a situation (cf. Proposition~\ref{pp:gd=}).

In Section~\ref{sec:Dynkin} and Section~\ref{sec:here},
we study the case of the bounded derived category $\DQ$ of an acyclic quiver $Q$.
In Theorem~\ref{thm:last}, we show that when $Q$ is a Dynkin quiver,
the minimal value $1-2/h$ of $\gldim$ is attained by
a unique solution of the Gepner equation \eqref{eq:GP}.
Moreover, in Theorem~\ref{thm:Dynkin}, we show that the image of
$\Stab\DQ/\CC$ under $\gldim$ is $[1-\frac{2}{h},+\infty)$ for a Dynkin quiver $Q$;
and in Theorem~\ref{thm:here}, we show that the image
of $\gldim$ is $[1,+\infty)$ for a non-Dynkin acyclic quiver $Q$.

In Section~\ref{sec:contractible}, we discuss the potential application of $\gldim$
to the contractibility conjecture for spaces of stability conditions.

In Appendix~\ref{sec:q}, we calculate the Gepner points in the spaces of $q$-stability conditions for Dynkin quivers.
\subsection*{Acknowledgments}
I would like to thank Tom Bridgeland, Akishi Ikeda, Bernhard Keller, Alastair King,
Kyoji Saito, Yukinobu Toda and Yu Zhou for inspirational discussions.
The idea was developed when I visited IPMU, Tokyo University in March and April 2018.
Also thanks to the anonymous referee who provides many useful suggestion to improve the exposition of the paper.
This work is supported by
National Key R\&D Program of China (No. 2020YFA0713000),
Beijing Natural Science Foundation (Z180003)
and Hong Kong RGC 14300817 (from the Chinese University of Hong Kong).

\section{Preliminaries}\label{sec:pre}
\subsection{Bridgeland stability conditions}\label{sec:BSC}
Following Bridgeland \cite{B1}, we recall the notion of stability conditions
on a triangulated category.
In this paper, $\D$ is a triangulated category with Grothendieck group
$K(\D)$ and we usually assume that $K(\D)\cong\ZZ^n$ for some $n$.
We denote by $\Ind\D$ the set of (isomorphism classes of) indecomposable
objects in $\D$.

\begin{definition}
\label{def:stab}
Let $\D$ be a triangulated category.
A {\it stability condition} $\sigma = (Z, \sli)$ on $\D$ consists of
a group homomorphism $Z \colon K(\D) \to \CC$ called the {\it central charge} and
a family of full additive subcategories $\sli (\phi) \subset \D$ for $\phi \in \R$
called the {\it slicing}
satisfying the following conditions:
\begin{itemize}
\item[(a)]
if  $0 \neq E \in \sli(\phi)$,
then $Z(E) = m(E) \exp(i \pi \phi)$ for some $m(E) \in \R_{>0}$,
\item[(b)]
for all $\phi \in \R$, $\sli(\phi + 1) = \sli(\phi)[1]$,
\item[(c)]if $\phi_1 > \phi_2$ and $A_i \in \sli(\phi_i)\,(i =1,2)$,
then $\Hom_{\D}(A_1,A_2) = 0$,
\item[(d)]for each object $0 \neq E \in \D$, there is a finite sequence of real numbers
\begin{equation}\label{eq:>}
\phi_1 > \phi_2 > \cdots > \phi_l
\end{equation}
and a collection of exact triangles (known as the \emph{HN-filtration})
\begin{equation*}
0 =
\xymatrix @C=5mm{
 E_0 \ar[rr]   &&  E_1 \ar[dl] \ar[rr] && E_2 \ar[dl]
 \ar[r] & \dots  \ar[r] & E_{l-1} \ar[rr] && E_l \ar[dl] \\
& A_1 \ar@{-->}[ul] && A_2 \ar@{-->}[ul] &&&& A_l \ar@{-->}[ul]
}
= E
\end{equation*}
with $A_i \in \sli(\phi_i)$ for all $i$.
\end{itemize}
The categories $\sli(\phi)$ are then abelian. Their non zero
objects are called {\it semistable of phase $\phi$} and their simple objects
{\it stable of phase $\phi$}.
For a semistable object $E\in\sli(\phi)$,
denote by $\phi_\sigma(E)\colon\!\!=\phi$ its phase.
Finally, for any interval $J\subset\R$, denote by $\sli(J)$ the subcategories consisting of
objects whose HN-filtrations have factors with phases in $J$.
\end{definition}

For a stability condition $\sigma = (Z,\sli)$,
we introduce the set of semistable classes
$\mathcal{C}^{\mathrm{ss}}(\sigma) \subset K(\D)$ by
\begin{equation*}
\mathcal{C}^{\mathrm{ss}}(\sigma) :=\{\,\alpha \in
K(\D)\,\vert\,\text{there is a semistable object }
E \in \D \text{ such that } [E] = \alpha\,\}.
\end{equation*}
We always assume that our stability conditions satisfy the following condition,
called the {\it support property} \cite{KoSo}.
Let $\norm{\,\cdot\,}$ be some norm on $K(\D) \otimes \R$.
A stability condition $\sigma=(Z,\sli)$
satisfies the support property if there is some constant $C >0$ such that
\begin{equation*}
C \cdot |{Z(\alpha)}| > \norm{\alpha}
\end{equation*}
for all $\alpha \in \mathcal{C}^{\mathrm{ss}}(\sigma)$.

There is a natural $\CC$-action
on the set $\Stab(\D)$ of all stability conditions on $\D$, namely:
\[
    s \cdot (Z,\hh{P})=(Z \cdot e^{-\mathbf{i} \pi s},\hh{P}_{\Re(s)}),
\]
where $\hh{P}_x(\phi)=\hh{P}(\phi+x)$.
There is also a natural action on $\Stab(\D)$ by the group of
autoequivalences $\Aut(\D)$, namely:
$$\Phi  (Z,\hh{P})=\big(Z \circ \Phi^{-1}, \Phi (\hh{P}) \big).$$

\subsection{Gepner equation}
Following Toda \cite{T1}, we introduce Gepner equations on
the space of stability conditions.
\begin{definition}[Toda]
A {\em Gepner equation} on $\Stab\D$ is
\begin{gather}\label{eq:Phi=s}
    \Phi(\sigma)=s \cdot \sigma,
\end{gather}
where $(\Phi,s)\in(\Aut\D)\times\CC$.

A {\em Gepner point} is a solution of some Gepner equation.
\end{definition}
Note that a Gepner point is an orbifold point of
$\Aut(\D)\backslash\Stab\D/\CC$.

\section{Global dimension functions}\label{sec:gldim}
\subsection{Global dimension functions}\label{sec:gdf}
Motivated by the study of Calabi-Yau-$\XX$ categories
and $q$-stability conditions in \cite{IQ1},
we introduce a function on (certain quotient spaces of)
spaces of stability conditions. It is to be viewed as a real-valued
generalization of the usual notion of global dimension
for algebras or abelian categories.

\begin{definition}
Let $\hh{P}$ be a slicing on a triangulated category $\D$.
Define the {\em global dimension of $\hh{P}$} by
\begin{gather}\label{eq:geq}
\gldim\hh{P}=\sup\{ \phi_2-\phi_1 \mid
    \Hom(\hh{P}(\phi_1),\hh{P}(\phi_2))\neq0\}\in\R_{\ge0}\cup\{+\infty\}.
\end{gather}
The global dimension of a stability condition $\sigma=(Z,\hua{P})$
is defined to be $\gldim\hua{P}$ for its slicing.
We say that $\hua{P}$ (or $\sigma)$ is {\em $\gldim$-reachable} if
there are phases $\phi_1$ and $\phi_2$ such that
\[
    \Hom(\hh{P}(\phi_1),\hh{P}(\phi_2))\neq0
    \quad\text{and}\quad \gldim\hh{P}=\phi_2-\phi_1.
\]
\end{definition}
For an algebra $A$, let $\hua{P}_A$ be the {\em canonical slicing}
on $\D^b(A)$, i.e. we have $\hua{P}_A(0)=\mod A$ and $\hua{P}_A(0,1)=\emptyset$.
Then we have $\gldim A=\gldim\hua{P}_A$.
We have the following facts (cf. \cite[\S~2]{IQ1}).
For any triangulated category $\D$,
\begin{itemize}
\item $\gldim$ is a continuous function on $\Stab\D$ (\cite[Lem.~2.23]{IQ1}).
\item $\gldim$ is an invariant under the $\CC$-action and the action of $\Aut\D$.
\end{itemize}
Thus, we have a function
\begin{gather}\label{eq:gldim}
    \gldim\colon\Aut\backslash\Stab\D/\CC\to\R_{\ge0}.
\end{gather}
Here $\R_{\ge0}$ includes $+\infty$.
We will sometimes also view $\gldim$ as a function on $\Stab\D$ or on $\Stab\D/\CC$.

\subsection{Total (semi)stability}
We recall the notion of total (semi)stability,
which is motivated by \cite[Conjecture~7.13]{Q2} due to Reineke \cite{R}.

\begin{definition}\cite{Q2}
A stability condition $\sigma$ on $\D$ is {\em totally (semi)stable},
if every indecomposable object is (semi)stable with respect to $\sigma$.
We will call such $\sigma$ a total stability condition.
\end{definition}

\begin{lemma}\label{lem:=>}
If $\gldim\sigma\leq1$, then $\sigma$ is totally semistable.
Moreover, if the inequality is strict, then $\sigma$ is totally stable.
\end{lemma}
\begin{proof}
Suppose that $E\in\Ind\D$ is not totally semistable. Then it admits an HN-filtration \eqref{eq:>} for some $l>1$.
So there is a triangle $A_1\to E\to E'\xrightarrow{\alpha} A_1[1]$ such that
$\alpha\in\Hom(E',A_1[1])$ does not vanish and $E'$ admits a filtration
with factors $(A_2,\ldots,A_l)$.
Therefore, we have $\Hom(A_k,A_1[1])\neq0$ for some $2\le k\le l$,
which implies
\begin{gather}\label{eq:gd>1}
    \gldim\sigma\geq\phi_\sigma(A_1[1])-\phi_\sigma(A_k)=\phi_1+1-\phi_k>1.
\end{gather}
Thus the first assertion holds.

For the second assertion, take a semistable but not stable object $E\in\sli(\phi)$.
Then it admits a filtration with factors in $\sli(\phi)$ (i.e. all stable subquotients
have phase $\phi$).
Applying the same argument as above we get the same calculation as
in \eqref{eq:gd>1} with $\ge 1$ at the end, which implies the second assertion.
\end{proof}

\begin{lemma}\label{lem:<=}
If $\sigma$ is totally semistable, then $\gldim\sigma\le1$.
Moreover, if $\sigma$ is totally stable and $\gldim$-reachable, then $\gldim\sigma<1$.
\end{lemma}
\begin{proof}
Suppose that $\gldim\sigma>1$, then there are semistable objects $E_1$ and $E_2[1]$
with phase $\phi_i=\phi_\sigma(E_i)$ such that
\[
    \Hom(E_1,E_2[1])\neq0, \quad \phi_2-\phi_1>0.
\]
Hence there is an object $X$ that sits in the triangle
\begin{gather}\label{eq:ses X}
    E_2\to X \to E_1 \xrightarrow{\alpha} E_2[1]
\end{gather}
with $\alpha\neq0$.

By the uniqueness of the HN-filtration, \eqref{eq:ses X} is the HN-filtration of $X$.
As the filtration is of length greater than one
the object X is non-semistable, hence decomposible (by total semistability).
By assumption, the indecomposible factors of $X$ are all semistable.
Hence they form a splitting HN-filtration of $X$ different from that in \eqref{eq:ses X}.
This contradicts the uniqueness of the HN-filtration.

Similarly for the second statement. The only difference is that when $\sigma$ is totally stable,
one can only deduce $\phi_2-\phi_1<1$ for any $\Hom(M_1,M_2)\neq0, M_i\in\hh{P}(\phi_i)$.
So we need the gldim-reachable assumption to get $\gldim\sigma<1$.
\end{proof}

Combining the lemmas above, we have the following.
\begin{proposition}\label{pp:leq}
$\sigma$ is totally semistable if and only if $\gldim\sigma\leq1$.
Moreover, when $\sigma$ is $\gldim$-reachable,
then $\sigma$ is totally stable if and only if $\gldim\sigma<1$.
\end{proposition}

\def\Sss{\operatorname{ToSt}(A_n)}
\subsection{Totally stable stability conditions for type A quiver}\label{sec:A}
Let us describe all total stability conditions for type A quiver and
give explicit formula of global dimension function in such a case.
Note that the $\CC$-actions preserve total stability.
Consider the $A_n$ quiver with straight orientation
\begin{gather}\label{eq:An}
    Q=\overrightarrow{A_n}\colon 1\leftarrow 2\leftarrow \cdots \leftarrow n.
\end{gather}
For $1\le i\le j\le n$,
denote by $M_{ij}$ the indecomposable object in $\mod \k Q$ that corresponds to
the representation $\{V_k,f_a\mid k\in Q_0, a\in Q_1\}$ of $Q$
with $V_k=\k$ if $i\le k\le j$ and $V_k=0$ otherwise.
So $S_i=M_{ii}$ are the simple objects in $\mod \k Q$.
The Auslander-Reiten quiver of $\mod \k Q$ is illustrated as follows.
\begin{gather}\label{eq:AR An}
\def\xquan{circle(.5)}
\begin{tikzpicture}[xscale=.7,yscale=-.7,rotate=-45]
\clip[rotate=-45](-.8,2)rectangle(6.3,15.5);
\foreach \j in {-3,...,8}{
\foreach \k in {0,...,8}{
    \draw[teal,->,>=stealth,very thick]
        (2*\k,2*\j+.7)to++(0,.6);
    \draw[teal,->,>=stealth,very thick]
        (2*\k+.7,2*\j)to++(.6,0);
}}
\draw[rotate=-45,white,fill=white](0,2)rectangle(-2,19)(8,2)rectangle(6,19);
\foreach \j in {-3,...,8}{
\foreach \k in {0,...,8}{
    \draw[white](2*\k,2*\j) node[circle,draw=teal] (x\k\j) {$00$};
}}
\foreach \j in {1,...,5}{
\foreach \k in {\j,...,5}{
    \draw[](2*\k,2*\j)node[circle,draw=teal,fill=teal!10] {$\j\k$};
}}
\end{tikzpicture}
\end{gather}
Denote by $\Sss$ the subspace in $\Stab\D_\infty(A_n)$
consisting of total stability conditions.
Denote by $\Poly(m)$ be the set of convex $m$-gon on $\R^2=\CC$
with vertices $P_0,P_1,\ldots,P_m=P_0$ in anticlockwise order satisfying $P_0=0$ and $P_1$=1.

\begin{proposition}\label{pp:A}
There is a natural bijection
\[\mathfrak{Z}\colon\Sss/\CC\to\Poly(n+1),\]
sending a stability condition $\sigma$ to an $(n+1)$-gon $\mathbf{P}$.
such that the oriented diagonals of $\mathbf{P}$ gives the central charges, with respect to $\sigma$,
of indecomposable objects in $\D_\infty(A_n)$.
\end{proposition}
\begin{figure}[ht]\centering
\begin{tikzpicture}[xscale=.7,yscale=.5]
\path (-1,0) coordinate (t1) node[below]{$P_0$};
\path (1,0) coordinate (t2) node[below]{$P_1$};
\path (3,3) coordinate (t3) node[right]{$P_2$};
\path (3.5,7) coordinate (t4)node[right]{$P_3$};
\path (-1,10) coordinate (t5)node[above]{$P_4$};
\path (-4,5) coordinate (t6)node[left]{$P_5$};
\draw(8,5)node{$Z(M_{ij})=\overrightarrow{P_{i-1}P_j}$};
\draw(8,4)node[below]{$1\le i\le j\le 5$};
\draw[red](8,1)node[below]{$\alpha=\arg\overrightarrow{P_3P_0}-\arg\overrightarrow{P_1P_4}$};
\foreach \j/\i in {1/4,2/5}
{\draw[red] (t\j) to (t\i);}
\foreach \j in {1,2,3,4,5,6}
{\draw (t\j) node {$\bullet$};}
\foreach \j/\i in {1/2,2/3,3/4,4/5,5/6}
{\draw[->,>=stealth,very thick,blue!50] (t\j) to (t\i);}
\draw[dashed,thick,blue!50] (t6) to (t1);
\draw[blue!50]($(t1)!.5!(t2)$) node[below] {$Z_1$};
\draw[blue!50]($(t3)!.5!(t2)$) node[below right] {$Z_2$};
\draw[blue!50]($(t3)!.5!(t4)$) node[right] {$Z_3$};
\draw[blue!50]($(t5)!.5!(t4)$) node[above] {$Z_4$};
\draw[blue!50]($(t5)!.5!(t6)$) node[above left] {$Z_5$};

\draw[red,->,>=stealth] ($(t1)!.3!(t4)$) to[bend left] ($(t2)!.3!(t5)$);
\draw[red]($(t1)!.25!(t5)$) node[right]{$\quad\alpha$};
\end{tikzpicture}
\caption{Convex hexagon for a total stability condition on $\D_\infty(A_5)$}\label{fig:tt}
\end{figure}
\begin{proof}
Given a total stability condition $\sigma$ on $\D_\infty(A_n)$,
we can assume $Z(S_1)=1$ with phase $\phi_\sigma(S_1)=0$ up to the $\CC$-action.
Then $\sigma$ is uniquely determined by $Z_i=Z(S_i)$ as we have
$$Z(M_{ij})=Z_i+\cdots+Z_j.$$
Let $P_i$ be the point in $\R^2$ that corresponds to $Z(M_{1i}), 1\le i\le n$.
We claim that $\{P_i\mid 0\le i \le n\}$ forms a convex polygon $\mathbf{P}$.
To see this, we only need to show that the angle $P_{i-1}P_iP_{i+1}$ is in $(0,\pi)$.
This is equivalent to
\begin{gather}\label{eq:123}
    \begin{cases}
    \phi_\sigma(M_{1n}[-1])<\phi_\sigma(S_1)<\phi_\sigma(M_{1n}[-1])+1,\\
    \phi_\sigma(S_i)<\phi_\sigma(S_{i+1})<\phi_\sigma(S_i)+1,&1\le i\le n,\\
    \phi_\sigma(S_n)<\phi_\sigma(M_{1n}[1])<\phi_\sigma(S_n)+1,
\end{cases}
\end{gather}
since $Z_i=\overrightarrow{P_{i-1}P_i}$.
Now, as there are non-zero morphisms in
\[\begin{cases}
    \Hom(M_{1n}[-1],M_{2n}[-1]),&\Hom(M_{2n}[-1],S_{1})\,\quad\text{and}\quad \Hom(S_{1},M_{1n}),\\
    \Hom(S_i,M_{i,i+1}),&\Hom(M_{i,i+1},S_{i+1})\quad\;\text{and}\quad \Hom(S_{i+1},S_i[1]),\\
    \Hom(S_n,M_{1,n-1}[1]),&\Hom(M_{1,n-1},M_{1n})\quad\text{and}\quad \Hom(M_{1n},S_n),\\
\end{cases}\]
where all the objects are stable, we have \eqref{eq:123}.
Hence we obtain an injective map $\mathfrak{Z}\colon\Sss/\CC\to\Poly(n+1)$.

What is left to show, is that $\mathfrak{Z}$ is surjective.
Given a convex $(n+1)$-gon $\mathbf{P}$ with vertices $P_i$,
we claim that there is a total stability condition $\sigma=(Z,\hh{P})$
satisfying
\begin{gather}\label{eq:Z phi}
\begin{cases}
    Z(M_{ij})=\overrightarrow{P_{i-1}P_j},\\
    \phi_\sigma(M_{ij})=(\arg \overrightarrow{P_{i-1}P_j})/\pi,
\end{cases}\end{gather}
where $\arg$ takes values in $[0,2\pi)$.
In fact, we only need to check that for any arrow $E\to F$
in the AR-quiver of $\D_\infty(A_n)$, one has $\phi_\sigma(E)<\phi_\sigma(F)$.
There are three types of such arrows (up to shift):
\[
    M_{ij}\to M_{i+1,j},\quad M_{ij}\to M_{i,j+1}\quad\text{and}\quad M_{kn}\to M_{1,k-1}[1],
\]
for $1\le i\le j\le n$ and $2\le k\le n$.
But
\[\begin{cases}
    (\arg \overrightarrow{P_{i-1}P_j})/\pi<(\arg \overrightarrow{P_{i}P_j})/\pi,\\
    (\arg \overrightarrow{P_{i-1}P_j})/\pi<(\arg \overrightarrow{P_{i-1}P_{j+1}})/\pi\quad\text{and}\\
    (\arg \overrightarrow{P_{k-1}P_n})/\pi<(\arg \overrightarrow{P_{0}P_{k-1}})/\pi+1\\
\end{cases}\]
indeed follow from the convexity of $\mathbf{P}$.
\end{proof}

\begin{remark}
Note that in the proof of the proposition above,
the heart of the total stability condition $\sigma$ is usually NOT $\mod \k Q$
for $Q=\overrightarrow{A_n}$ in \eqref{eq:An}.
Moreover, given any orientation of a quiver $Q$ of type $A_n$,
there always exits a total stability condition $\sigma$ with heart $\mod \k Q$ \cite{B}.
But this is not our focus.
\end{remark}

\begin{proposition}\label{pp:gd=}
Let $\sigma\in\Sss$ and $\mathbf{P}=\mathfrak{Z}(\sigma)$ be the corresponding $(n+1)$-gon.
Then
\begin{gather}\label{eq:gldim a}
    \gldim\sigma=\frac{1}{\pi}\max\{
    \arg\overrightarrow{P_jP_i}-\arg\overrightarrow{P_{i+1}P_{j+1}}
        \mid 0\le i<j\leq n \}
\end{gather}
where $P_{n+1}=P_0$ (cf. angle $\alpha$ in Figure~\ref{fig:tt} and angles in Figure~\ref{fig:7}).
\end{proposition}
\begin{proof}
By Auslander-Reiten duality, we have $\Hom(E,F)=D\Hom(F,\tau(E[1]))$,
where $\tau$ is the Auslander-Reiten functor.
As any indecomposable is stable, we have
\[
    \gldim\sigma=\max\{ \phi(\tau(E[1])-\phi(E) \mid \text{indecomposable $E$} \}.
\]
On the other hand, any indecomposable $E$ in $\D_\infty(A_n)$ is of the form $M_{ij}[m]$
for $1\le i\le j\le n$ and $m\in\ZZ$.
Then formula above becomes \eqref{eq:gldim a} using \eqref{eq:Z phi}.
\end{proof}
\subsection{Calabi-Yau case}
\begin{definition}
For a stability condition $\sigma$,
we say that an auto-equivalence $\Phi$ is $\sigma$-semistable if $\Phi(E)$ is $\sigma$-semistable
for every $\sigma$-semistable object $E$. If in addition
$\phi_\sigma(\Phi(E))-\phi_\sigma(E)=s$
for every $\sigma$-semistable object $E$,
then we say that $\phi_\sigma(\Phi)$ exists and equals $s$.
\end{definition}

Clearly, the shift functor $[1]$ is semistable with respect to any stability condition.
Denote by $\kong{S}$ the Serre functor of $\D$ (provided it exists),
that is the auto-equivalence in $\Aut\D$ satisfying
\begin{equation}\label{eq:serre}
    \mathrm{S}\colon\Hom(L,M)
        \xrightarrow{\sim}  D\Hom (M,\kong{S}(L)).
\end{equation}

\begin{proposition}\label{pp:CY case}
Let $\sigma\in\Stab\D$.
If $\phi_\sigma(\kong{S})$ exists,
then $\gldim\sigma=\phi_\sigma(\kong{S})$.
\end{proposition}
\begin{proof}
Since
\[
    0\ne D\Hom(E,E)=\Hom(E, \kong{S}(E))
\]
for any (semi)stable object $E$,
then
\[
    \phi_{\sigma}(\kong{S}(E))-\phi_{\sigma}(E)=\phi_\sigma(\kong{S})
\]
which implies $\gldim\sigma\ge \phi_\sigma(\kong{S})$.

On the other hand, for any semistable object $E'$
such that $\Hom(E,E')\ne0$, we have
\[
    \Hom(E',\kong{S}(E))=D\Hom(E,E')\neq0,
\]
which implies $\phi_{\sigma}(E')\le\phi_{\sigma}(\kong{S}(E))$.
Thus,
\[
    \phi_{\sigma}(E')-\phi_{\sigma}(E)\leq
    \phi_{\sigma}(\kong{S}(E))-\phi_{\sigma}(E)=\phi_\sigma(\kong{S})
\]
and $\gldim\sigma\le\phi_\sigma(\kong{S})$.
This completes the proof.
\end{proof}

We have the following immediate consequence.
\begin{corollary}\label{cor:CYN}
If $\hua{D}$ is Calabi-Yau-$N$ for some integer $N\ge 1$,
i.e. $\kong{S}=[N]$,
then $\phi_\sigma(\kong{S})=N$ for any $\sigma$.
Thus $\gldim\equiv N$ on $\Stab\D$.
\end{corollary}
\begin{proof}
Since $\phi(E[1])=\phi(E)+1$,
$\phi_\sigma(\kong{S})$ exists and equals $N$.
Then the assertion follows from the proposition above.
\end{proof}

A remark is that Serre-invariant stability conditions are studied in \cite{FP,PR},
where the uniqueness of such stability conditions is proven in certain cases.

\section{Global dimension for Dynkin quiver}\label{sec:Dynkin}
\subsection{Homotopy category of graded matrix factorizations}
Consider a weighted graded polynomial ring
\begin{gather}\label{eq:A}
    A \colon= \CC[x_1, x_2, \cdots, x_m], \quad \deg x_i= a_i
\end{gather}
with $a_1 \ge a_2 \ge \cdots \ge a_m$.
Let $f \in A$ be a homogeneous element of degree $h$, known as superpotential
such that $(f=0) \subset \CC^n$ has an isolated singularity at the origin.
Denote by $P(k)$ the $k$-grading shifts of a graded $A$-module $P=\oplus_{i\in\ZZ} P_i$,
where $P(k)_i= P_{i+k}$.
\begin{definition}
A \emph{graded matrix factorization} of $f$ is a $\ZZ_2$-graded complex
\begin{gather}\label{eq:MF}
P^\bullet\colon\quad P^0 \stackrel{\,p^0\,}{\longrightarrow}  P^1 \stackrel{\,p^1\,}{\longrightarrow} P^0(h)
\end{gather}
where $P^i$ are graded free $A$-modules of finite rank and $p^i$
are homomorphisms of graded $A$-modules, satisfying the following conditions:
\begin{gather*}
\quad p^1 \circ p^0= f\cdot \id, \quad
p^0(h) \circ p^1= f\cdot \id.
\end{gather*}

\end{definition}
The category $\HMF(f)$ is defined to be the homotopy category of graded matrix factorizations of $f$:
whose objects consist of $\ZZ_2$-graded complexes \eqref{eq:MF} and
the set of morphisms are given by the commutative diagrams of graded $A$-modules
modulo null-homotopic morphisms.
The category $\HMF(f)$ is a triangulated category where
the shift $[1]$ of \eqref{eq:MF} is
\begin{gather}
    P^\bullet[1]\colon P^1 \stackrel{-p^1}{\longrightarrow} P^0(h) \xrightarrow{-p^0(h)} P^1(h)
\end{gather}
and the distinguished triangles are defined via the
usual mapping cone constructions, cf. \cite{KST, T1}.
The grading shift functor
$P^{\bullet} \mapsto P^{\bullet}(1)$
induces the autoequivalence $\RA$ of $\HMF(f)$ such that
\begin{align}\label{eq:RA}
    \RA^h =[2].
\end{align}
Furthermore,
there exists a Serre functor $\kong{S}_f$ on $\HMF(f)$:
\begin{align}\label{eq:Serre}
\kong{S}_f = \RA^{-\varepsilon}[m-2]
\end{align}
where $\varepsilon$ is the \textit{Gorenstein index} defined by
\begin{gather}
\varepsilon \colon= \sum_{i=1}^{m} a_i -h \in \mathbb{Z}.
\end{gather}

\begin{remark}
We reserve $\AR$ for Auslander-Reiten functor.
\end{remark}

\subsection{Toda's conjecture}
Consider a central charge $Z_G$ on $\HMF(f)$, symbolically defined by
\begin{align}\label{eq:ZG}
    Z_G(P^\bullet)\colon=\mathrm{sTr}(e^{2\mathbf{i} \pi /h} \colon P^{\bullet} \to P^{\bullet})
\end{align}
for a graded matrix factorization $P^\bullet$ in \eqref{eq:MF}.
Here the $e^{2\pi \mathbf{i}/h}$-action is induced by the $\mathbb{Z}$-grading on each $P^i$
and $\mathrm{sTr}$ means the super trace of the $e^{2\pi \mathbf{i}/h}$-action
with respect to the $\mathbb{Z}/2\mathbb{Z}$-grading of $P^{\bullet}=P^0 \oplus P^1$.

\begin{conjecture}\label{conj:Toda}\cite{T1}
There is a stability condition $\sigma_G=(Z_G,\hh{P}_G)$ whose central charge is given by \eqref{eq:ZG}.
Moreover, it satisfies the Gepner equation with respect to $(\RA, 2/h)$ in $\Stab\HMF(f)$.
In other words, the equation
\begin{gather}\label{eq:RA=}
    \RA(\sigma)=\frac{2}{h} \cdot \sigma
\end{gather}
admits a solution $\CC\cdot\sigma_G$.
\end{conjecture}

\subsection{Kajiura-Saito-Takahashi's solution in the Dynkin case}
In this section, we focus on $m=3, \varepsilon>0$ case
where Kajiura-Saito-Takahashi \cite{KST} proved the triangle equivalence \eqref{eq:KST}
and Conjecture~\ref{conj:Toda}.

Let $Q$ be a Dynkin quiver, i.e. the underlying diagram of $Q$ is one of the following graphs
(with the labeling of the vertices)
\begin{gather}\label{eq:labeling}
\begin{array}{llr}
    \xymatrix{A_{n}:} \quad &
    \xymatrix{
        1 \ar@{-}[r]& 2 \ar@{-}[r]& \cdots \ar@{-}[r]& n }         \\
    \xymatrix@R=0.1pc{\\D_{n}:} \quad &
    \xymatrix@R=0.1pc{
        &&&& n-1 \\
        1 \ar@{-}[r]& 2 \ar@{-}[r]& \cdots \ar@{-}[r]& n-2 \ar@{-}[dr]\ar@{-}[ur]\\
        &&&& n &}                                                             \\
    \xymatrix{E_{6,7,8}:} \quad &
    \xymatrix@R=1pc{
        && 4 \\
        1 \ar@{-}[r]& 2 \ar@{-}[r]& 3 \ar@{-}[r]\ar@{-}[u]& 5 \ar@{-}[r]& 6
                \ar@{--}[r]& 7 \ar@{--}[r]& 8}
\end{array}
\end{gather}
Let $h=h_Q$ be the Coxeter number, which is
\[ h_Q=\begin{cases}
    n+1,    \quad &\text{if $Q$ if of type $A_n$};\\
    2(n-1),    \quad &\text{if $Q$ if of type $D_n$};\\
    12,18,30    \quad &\text{if $Q$ if of type $E_6,E_7,E_8$ respectively.};
\end{cases}\]
Denote by $\DQ=\D^b(\k Q)$ the bounded derived category of the path algebra $\k Q$
and by $\AR$ its Auslander-Reiten functor.
Let $(x,y,z)=(x_1,x_2,x_3)$, $(a,b,c)=(a_1,a_2,a_3)$.

\begin{theorem}\cite{KST}\label{thm:KST}
There are triangle equivalences
\begin{gather}\label{eq:KST}
    \Phi\colon\HMF(f)\cong\DQ,
\end{gather}
for $m=3, \varepsilon>0$,
\[
    (a,b,c;\,h)=\begin{cases}
        (1,b,n+1-b;\,n+1),& Q=A_n (n\ge1, 1\le b\le n);\\
        (n-2,2,n-1;\,2(n+1)),& Q=D_n, (n\ge4);\\
        (4,3,6;\,12), & Q=E_6;\\
        (6,4,9;\,18),& Q=E_7;\\
        (10,6,15;\,30),& Q=E_8;
   \end{cases}
\]
and
\[
    f(x,y,z)=\begin{cases}
        x^{n+1}+yz,& Q=A_n (n\ge1);\\
        x^2y+y^{n-1}+z^2,& Q=D_n, (n\ge4);\\
        x^3+y^4+z^2,& Q=E_6;\\
        x^3+x y^3+z^2,& Q=E_7;\\
        x^3+y^5+z^2,& Q=E_8;
   \end{cases}
\]
Moreover, $\Phi(\RA)=\AR^{-1}$,
the Serre functor in $\DQ$ is
\begin{gather}\label{eq:s ar h}
\kong{S}=[1]\circ\AR\quad\text{with}\quad\AR^h=[-2]\Leftrightarrow\kong{S}^{h}=[h-2].
\end{gather}
and
Conjecture~\ref{conj:Toda} holds.
\end{theorem}

\begin{remark}
When $Q$ is of type $A_n$,
the corresponding polygon is a regular $(n+1)$-gon
(see Proposition~\ref{pp:A})
and any angle $\arg\overrightarrow{P_jP_i}-\arg\overrightarrow{P_{i+1}P_{j+1}}$
in \eqref{eq:gldim a} equals $\frac{n-1}{n+1}\pi$, cf. Figure~\ref{fig:7} for $n=6$ case.
\end{remark}

\begin{figure}[ht]\centering
\begin{tikzpicture}[xscale=.7,yscale=.7,rotate=90]
\foreach \j in {1,...,7}{
\draw[blue!5,very thick] (360/7*\j:4) to (360/7*\j+360/7*3:4);}
\foreach \j in {1,...,7}{
\draw[blue!40,thick] (360/7*\j:4) to (360/7*\j+360/7*2:4);}
\foreach \j in {1,...,7}{
\draw[blue!99,thick] (360/7*\j:4) to (360/7*\j+360/7*1:4);}
\foreach \j in {1,...,7}{
\draw[red,->,>=stealth] (360/7*\j-5:3.82) to[bend left] (360/7*\j+5:3.82);
\draw[red,->,>=stealth] (360/7*\j-12:.9) to[bend left] (360/7*\j+12:.9);
\draw[red,->,>=stealth] (360/7*\j-7+360/14:2.63) to[bend left] (360/7*\j+7+360/14:2.63);}
\end{tikzpicture}
\caption{Regular heptagon for $\sigma_G$ in $A_6$ case and angles equals $\pi\gldim\sigma_G$}\label{fig:7}
\end{figure}

\subsection{Uniqueness}
Define
\begin{gather}
    \cd=1-\frac{2}{h}.
\end{gather}
We will show that $\gldim\DQ=\cd$.
Moreover, we will prove the uniqueness of the solution of \eqref{eq:RA=} (cf. \cite[\S~2.9]{T1} for the uniqueness)
and the uniqueness of the solution of equation $\gldim\sigma=\cd$, up to the $\CC$-action.

\begin{lemma}\label{lem:gldim G}
$\gldim\sigma_G=\cd$ for a Dynkin quiver $Q$.
\end{lemma}
\begin{proof}
Since $\AR(\sigma_G)=\RA^{-1}(\sigma_G)=(-2/h)\cdot\sigma_G$,
we have $\kong{S}(\sigma_G)=\cd\cdot\sigma_G$,
which implies that $\phi_{\sigma_G}(\kong{S})=\cd$.
Then the lemma follows from Proposition~\ref{pp:CY case}.
\end{proof}

\begin{theorem}\label{thm:last}
Up to the $\CC$-action, $\sigma_G$ is the unique solution of \eqref{eq:RA=}, or equivalently,
\begin{gather}\label{eq:GP}
    \tau(\sigma)=(-\frac{2}{h})\cdot\sigma.
\end{gather}
Moreover, the minimal value of $\gldim$ on $\Stab\DQ$ is $\cd$
and $\sigma_G$ is also the unique solution of $\gldim\sigma=\cd$ up to the $\CC$-action.
\end{theorem}
\begin{proof}
First, we show that any solution $\sigma$ of \eqref{eq:GP} satisfies $\gldim\sigma\le\cd$.
In fact, for any two semistable objects $E, F$ with $\Hom(E,F)\ne 0$,
we have $\Hom(F,\kong{S}(E))\ne0$. On the other hand, \eqref{eq:GP} implies
that $\kong{S}(E)=\tau(E[1])$ is also semistable.
Thus we have
$$\phi_\sigma(E)\le \phi_\sigma(F)\le\phi_\sigma(\kong{S}(E))=\phi_\sigma(E)+1-\frac{2}{h},$$
which implies $\phi_\sigma(F)-\phi_\sigma(E)\le1-2/h$. Thus $\gldim\sigma\le\cd$.

Now consider any $\sigma=(Z,\hh{P})$ with $\gldim\sigma\le\cd$.
We only need to show that it must be $\sigma_G$, up to the $\CC$-action, to complete the proof.

Since $\gldim\sigma\le\cd<1$,
Proposition~\ref{pp:leq} says that any indecomposable object is stable with respect to $\sigma$.
For any $E\in\Ind\DQ$, we have
$$\Hom( \kong{S}^{i-1}(E), \kong{S}^{i}(E) )\neq0\quad\text{for $i=1,\ldots,h$.}$$
Thus we have
\[\begin{array}{rcl}
    h\cdot\gldim\sigma &\geq&\displaystyle{\sum_{i=1}^h}\left(
        \phi_\sigma( \kong{S}^i(E) )-\phi_\sigma( \kong{S}^{i-1}(E) )
    \right)\\
    &=& \phi_\sigma( \kong{S}^h(E) ) )-\phi_\sigma( E )\\
    &=& \phi_\sigma( E[h-2] ) )-\phi_\sigma( E )
    \\&=& h-2,
\end{array}\]
where we use \eqref{eq:s ar h}.
Hence for any $E$ we have $\phi_\sigma( \kong{S}(E) )-\phi_\sigma(E)=1-2/h$,
i.e.
\begin{gather}\label{eq:tau p}
    \kong{S}(\hh{P})=\hh{P}_{1-2/h}\quad\text{or}\quad\tau(\hh{P})=\hh{P}_{-2/h},
\end{gather}
where $\hh{P}_x(\phi)=\hh{P}(\phi+x)$.
Thus $\gldim\sigma=\cd$.
We process to show that $\sigma$ is unique up to the $\CC$-action and satisfies \eqref{eq:GP}.

By Proposition~\ref{pp:leq}, all indecomposables are stable with respect to $\sigma$.
Consider a leaf $l$ (vertex with only one arrow in or out) of $Q$,
which corresponds to a boundary $\tau$-orbit in the Auslander-Reiten quiver of $\DQ$.
For instance, in Section~\ref{sec:A}, the leaf 1 in \eqref{eq:An}
corresponds to the last (red) dashed line in \eqref{eq:AR An}.
Denote the objects in this $\tau$-orbit by $L_i,\;i\in\ZZ$ such that $\tau(L_{i+1})=L_i$
(and hence $L_{i+h}=L_i[2]$).
Denote the objects in its neighbour $\tau$-orbit by $K_i,\; i\in\ZZ$ such that $\tau(K_{i+1})=K_i$
so that there are Auslander-Reiten triangles
\[
    L_i\to K_i\to L_{i+1} \to L_i[1].
\]
By \eqref{eq:tau p}, we have
\begin{gather}\label{eq:=2/h}
    \phi_\sigma(L_{i+1})-\phi_\sigma(L_i)=2/h=\phi_\sigma(K_{i+1})-\phi_\sigma(K_i)
\end{gather}
for any $i\in\ZZ$.
The triangles above imply that
$Z(K_i)=Z(L_i)+Z(L_{i+1})$
and they are all stable, which means that
$\phi_\sigma(K_i)$ is in the interval $$(\phi_\sigma(L_i),\phi_\sigma(L_{i+1}))=(\phi_\sigma(L_{i}), \phi_\sigma(L_{i+1})+2/h ).$$
Therefore
\begin{gather}\label{eq:cases}\begin{cases}
    |Z(L_i)|>|Z(L_{i+1})| \Longleftrightarrow \phi_\sigma(K_i)<\phi_\sigma(L_{i+1})+1/h,\\
    |Z(L_i)|=|Z(L_{i+1})| \Longleftrightarrow \phi_\sigma(K_i)=\phi_\sigma(L_{i+1})+1/h,\\
    |Z(L_i)|<|Z(L_{i+1})| \Longleftrightarrow \phi_\sigma(K_i)>\phi_\sigma(L_{i+1})+1/h.
\end{cases}\end{gather}
and one of these cases holds for all $i$ noticing the periodicity \eqref{eq:=2/h}.
However, $Z(L_{i+h})=Z(L_i[2])=Z(L_i)$.
This forces that the second case in \eqref{eq:cases} occurs.
Thus $\{Z(L_i),\phi_\sigma(L_i)\}$ satisfies \eqref{eq:GP},
which also determines $\{Z(K_i),\phi_\sigma(K_i)\}$.
Note that $\{Z(K_i),\phi_\sigma(K_i)\}$ also satisfies \eqref{eq:GP},
and determines $\{Z(L_i),\phi_\sigma(L_i)\}$.

Recall that $L_i$ corresponds to a leaf $l$ in $Q$ and
then $K_i$ corresponds to an adjacent vertex $k$ in $Q$.
If $k$ is not a trivalent vertex, let $j\ne l$ be the other adjacent vertex of $k$.
It is straightforward to see that the central charges and phases of objects
in the $\tau$-orbit corresponding to $j$
are determined by $\{Z(L_i),\phi_\sigma(L_i)\}$ and satisfy \eqref{eq:GP}.
We can keep this process untill we reach a trivalent vertex of $Q$.

As we can apply this analysis to any leaves in $Q$, we deduce that
the central charges and phases of objects in any given $\tau$-orbit of
the Auslander-Reiten quiver of $\DQ$ satisfy \eqref{eq:GP} and
they determine $\sigma=(Z,\hh{P})$ uniquely.
Therefore, up to $\CC$-action, $\sigma$ is unique.
By Lemma~\ref{lem:gldim G}, we conclude that $\sigma=s\cdot\sigma_G$
for some $s\in\CC$.
\end{proof}

\subsection{Range of $\gldim$ in the Dynkin case}
We proceed to calculate the range of
the global dimension function \eqref{eq:gldim} on $\Stab\DQ$.
Note that as $\Ind\DQ/[1]$ is finite,
$\gldim\sigma<+\infty$ for any $\sigma$.

\begin{theorem}\label{thm:Dynkin}
The range of $\gldim$ on $\Stab\DQ/\CC$ is $[\cd, +\infty)$.
Moreover, (the orbit of) $\sigma_G$ is the unique point that achieves the minimal value $\cd$.
\end{theorem}
\begin{proof}
By Theorem~\ref{thm:last}, $\gldim$ has minimal value $\cd$,
which is achieved by a unique stability condition in $\Stab\DQ/\CC$.

Note that $\gldim$ is continuous (cf. Section~\ref{sec:gdf}) and $\Stab\DQ$ is connected (cf. e.g. appendix of \cite{Qiu}), we only need to show that
there exists $\sigma$ with $\gldim\sigma$ arbitrary big.
This follows from the following facts:
\begin{itemize}
\item A heart $\h$ in $\DQ$ and a central charge $Z$ on its simples with values
in the upper half plane
\[
    H=\{r\exp(\mathrm{i} \pi \theta)\mid r\in\kong{R}_{>0},0\leq \theta<1 \}
        \subset\kong{C}
\]
determine a stability condition $\sigma$.
\item By definition, an arrow $S\xrightarrow{d}S'$ with degree $d$ in the Ext-quiver of a heart $\h$ corresponds to an element of a basis of $\Hom(S,S'[k])$.
Thus, $$\gldim\sigma\ge \phi_\sigma(S'[k])-\phi_\sigma(S) >k-1$$ for any $\sigma$ with heart $\h$.
\item Recall from \cite[Def.~3.7 and Def.~5.11]{KQ} there is the notion of
iterated simple (HRS-)tilting.
By \cite[Thm.~5.9]{KQ} and \cite[Thm~A.6]{Q2}
one can iterated forward tilt a simple of any heart in $\DQ$.
Hence the corresponding Ext quiver of the heart $\h$ can have arbitrary large degrees.
\end{itemize}
\end{proof}

The fact that the $\gldim$ can be arbitrary large
can be also deduced from \cite[Lem.~3.19]{M} easily.

\begin{example}
In the case when $Q$ is $A_2$ quiver,
the description of $\Stab\D_\infty(A_2)$ is explicit in \cite{BQS, Qiu}.
The shadow area in Figure~\ref{fig:A2} is
the fundamental domain for $\Aut\backslash\Stab\D_\infty(A_2)/\CC$.
Note that the ratio of the central charges of two of the semistable objects
equals $e^{\pi\mathbf{i}z}$ for $z=x+y\mathbf{i}$ and
$$l_{\pm}=\{ z \mid  x\in(\frac{1}{2},\frac{2}{3}],
        y\pi=\mp\ln(-2\cos x\pi)\}$$
are the boundaries.
And we have the formula to calculate $\gldim\sigma$ via $x$-axis in the fundamental domain:
\[
    x=1-\gldim\sigma.
\]
\end{example}

\begin{figure}\centering
\begin{tikzpicture}[yscale=.35,xscale=.7]
\path (0,0) coordinate (O);
\path (4,0) coordinate (A);
\path (3+0.2,10) coordinate (m1);
\path (3+0.2,-10) coordinate (m2);
\draw[white, fill=gray!14] (m1)
    .. controls +(-90:3) and +(120:1) .. (A)
    .. controls +(-120:1) and +(90:3).. (m2)
   to [out=180,in=0] (-5,-10)
   .. controls +(90:3) and +(-90:1) ..  (-5,10) ;
\draw[white,thick] (m2) edge (-5,-10);
\draw[dotted,->,>=latex] (-6,0) -- (8,0) node[right]{$x=1-\gldim\sigma$};
\draw[dotted,->,>=latex] (0,-10) -- (0,10) node[above]{$y$};
\draw (0,0) node[below left]{$0$};
\path[dotted] (6,-10) edge (6,10);
\path[dotted] (3,-10) edge (3,10);
\draw[red,thick] (A) edge (6,0);
\draw[red,thick] (m1) .. controls +(-90:3) and +(120:1) .. (A);
\draw[red,thick] (A)  .. controls +(-120:1) and +(90:3).. (m2);
\draw (6,0) node[below right] {$1$};
\path (4.5,0) node (a) {};
\path (3.6,1.5) node (b) {};
\path (3.6,-1.5) node (c) {};
\path[->,>=stealth,orange, bend right](a) edge (b);
\path[->,>=stealth,orange, bend right] (b) edge (c);
\path[->,>=stealth,orange, bend right] (c) edge (a);
\path (4,6) node {$l_+$};\path (4,-6) node {$l_-$};
\draw[red] (4,0) node {{$\bullet$}};
\path (4,0) node[below] {{$\frac{2}{3}$}};
\path (-6,0) node[left] {$-\infty$};
\draw[fill=white] (6,0) node[white] {$\bullet$} node {$\circ$};
\end{tikzpicture}
\caption{$\Aut\backslash\Stab\D_\infty(A_2)/\CC = \big(\Stab\D_\infty(A_2)/\CC\big)/C_3$}\label{fig:A2}
\end{figure}

\section{Global dimension for hereditary case}\label{sec:here}
A triangulated category $\D$ is hereditary if it admits a heart $\h$ such that
\begin{equation}\label{eq:here}
    \Hom^{\ge2}(\h,\h)=0.
\end{equation}
We call a heart $\h$ of $\D$ hereditary if \eqref{eq:here} holds.
The following is straightforward.

\begin{lemma}\label{lem:le1}
Let $\h$ be a hereditary heart of $\D$.
If $\h$ is finite (i.e. is a length category with finitely many simple objects),
then $\gldim\D\le1$.
\end{lemma}
\begin{proof}
Take a stability condition $\sigma=(Z,\mathcal{P})$ with heart $\h=\mathcal{P}(0,1]$,
such that the central charge $Z(S)$ is a negative real number for any simple $S\in\Sim\h$.
In other words, we have $\h=\mathcal{P}(1)$. Thus
\[
    \mathcal{P}(\phi)\ne\emptyset \Longleftrightarrow \phi\in\ZZ.
\]
Noticing \eqref{eq:here},
we have $\Hom(\mathcal{P}(i),\mathcal{P}(j))\neq0$ if and only if $i,j\in\ZZ$ and $i\le j\le 1$.
Thus $\gldim\sigma\leq1$ and the lemma follows.
\end{proof}
\subsection{Acyclic quiver case}
\begin{theorem}\label{thm:here}
Suppose that $Q$ is an acyclic quiver, which is not of Dynkin type.
Then the range of $\gldim$ on $\Aut\backslash\Stab\DQ/\CC$ is $[1, +\infty)$.
\end{theorem}
\begin{proof}
$\DQ$ admits a hereditary heart $\h_Q\colon=\mod \k Q$, which is finite/algebraic.
By Lemma~\ref{lem:le1}, we see that $\min\gldim\le1$.
Now suppose that $\gldim\sigma\le1$ for some $\sigma\in\Stab\DQ$.
By Proposition~\ref{pp:leq}, every indecomposable is semistable.
However, when $Q$ is not of Dynkin type, there are indecomposable objects
with non-trivial self-extension, which implies $\gldim\ge1$.
Thus, $\gldim\sigma=1$.
Any value in $(1,+\infty)$ is reachable is similar to the proof in Theorem~\ref{thm:Dynkin}.
What is left to show is that $\gldim\sigma<+\infty$ for any $\sigma$.

Let $\sigma=(Z,\hua{P})\in\Stab\DQ$ with heart $\h=\mathcal{P}(0,1]$.
There exists an integer $L>>1$ such that any simple of $\h_Q$ and hence $\h_Q$ is in $\hua{P}(-L,L)$.
Then we deduce that $\h\cap\h_Q[L']=\emptyset$ for integer $|L'|\ge L$.
As $$\Ind\DQ=\bigcup_{m\in\ZZ}\Ind\h_Q[m],$$
we deduce that
\[
    \h\subset \< \h_Q[m]\mid |m|<L \>.
\]
Thus \eqref{eq:here} imlies
\[
    \Hom^{>2L}(\h,\h)=0
\]
and hence $\gldim\sigma\le2L$.
\end{proof}

\subsection{Example: Kronecker quiver}\label{sec:P1}
Notice that we have a derived equivalence
\begin{gather}\label{eq:KQ}
    \D^b(\coh\mathbb{P}^1) \cong \D_\infty(K_2),
\end{gather}
where $K_2$ is the Kronecker quiver (two vertices with double arrows).
Similar to $A_2$ case (cf. Figure~\ref{fig:A2}), we have the following
fundamental domain $\hua{K}_0$ in Figure~\ref{fig:K2}
for $\Aut\backslash\Stab\D^b(\coh \mathbb{P}^1)/\CC$
(see \cite{Ok} and \cite[\S~7.5.2]{Qiu}).
Note that the ratio of the central charges of two of the semistable objects
equals $e^{\pi\mathbf{i}z}$ for $z=x+y\mathbf{i}$ and
\begin{equation}
\begin{cases}
    k_1=\{z=x+y\mathbf{i} \mid  x\in(\frac{1}{2},1),
        y\pi=\ln(-\cos x\pi), \\
    k_0=\{z=x+y\mathbf{i} \mid  x\in(\frac{1}{2},1),
        y\pi=-\ln(-\cos x\pi),
\end{cases}
\end{equation}
are the boundaries lines of $\hh{K}_0$ in the figure.
One can calculate the following.
\begin{figure}\centering
\begin{tikzpicture}[scale=1]
\path (0,0) coordinate (O);
\path (4,0) coordinate (A);
\path (3+0.2,5) coordinate (m3);
\path (3+0.2,-5) coordinate (m4);
\draw[fill=gray!14,dotted]   (m3)
    .. controls +(-90:4) and +(180:3) .. (6,0)
    .. controls +(180:3) and +(90:4) .. (m4)
    to [out=180,in=0] (-3,-5)
    to [out=90,in=-90] (-3,5);
\draw[white, thick]   (m4) to (-3,-5) to (-3,5);
\path (-2,3) node {$\hua{K}_0$};
\draw[dotted,->,>=latex] (6,0) -- (8,0) node[right]{$x$};
\draw[blue,->,>=latex] (0,-5) -- (0,5) node[above]{$y$};
\draw (0,0) node[below left]{$0$};
\path[dotted] (-2,0) edge (4,0);
\path[dashed] (6,-5) edge (6,5);
\path (3,-5)[dotted] edge (3,5);
\draw[NavyBlue,thick] (m3) .. controls +(-90:4) and +(180:3) .. (6,0);
\draw[NavyBlue,thick] (m4) .. controls +(90:4) and +(180:3) .. (6,0);
\path (3.6,4) node {$k_1$};\path (3.6,-4) node {$k_0$};
\draw[NavyBlue,dashed,thick] (6,0) .. controls +(180:2.8) and +(180:2) .. (6,3);
\draw[NavyBlue,dashed,thick] (6,0) .. controls +(180:2.8) and +(180:2) .. (6,-3);
\draw[NavyBlue,dashed,thick] (6,0) .. controls +(180:1.8) and +(180:2) .. (6,2);
\draw[NavyBlue,dashed,thick] (6,0) .. controls +(180:1.8) and +(180:2) .. (6,-2);
\path (6,0) node[below right] {$1$};
%
%
\draw[fill=black] (0,0) circle (.05);\draw[fill=white] (6,0) circle (.1);
\draw[fill=white] (6,3) circle (.1);\draw[fill=white] (6,-3) circle (.1);
\draw[fill=white] (6,2) circle (.1);\draw[fill=white] (6,-2) circle (.1);
\path (6,5) node[above] {\tiny{$\infty$}};
\path (7,5) node {\small{$\frac{1}{\pi} \ln\frac{1}{0}$}};
\path (6,3) node[right] {\small{$\frac{1}{\pi} \ln\frac{2}{1}$}};
\path (6,2) node[right] {\small{$\frac{1}{\pi} \ln\frac{3}{2}$}};
\path (6,1.324) node[right] {{$\vdots$}};
\path (6,-5) node[below] {\tiny{$-\infty$}};
\path (7,-5) node {\small{$\frac{1}{\pi} \ln\frac{0}{1}$}};
\path (6,-3) node[right] {\small{$\frac{1}{\pi} \ln\frac{1}{2}$}};
\path (6,-2) node[right] {\small{$\frac{1}{\pi} \ln\frac{2}{3}$}};
\path (6,-1.324) node[right] {{$\vdots$}};
\path (6,0) node (a1) {\tiny{}};
\path (4.7,1.5) node (a2) {\tiny};
\path (4.2,1.5) node (a3) {\tiny};
\path (3.4,1) node (a4) {\tiny};
\path (4.7,-1.5) node (a7) {\tiny};
\path (4.2,-1.5) node (a6) {\tiny};
\path (3.4,-1) node (a5) {\tiny};
\path[->,orange, bend right, dashed] (a1) edge (a2);
\path[->,orange, thick] (a2) edge (a3);
\path[->,orange, bend right, thick] (a3) edge (a4);
\path[->,orange, bend right, thick] (a4) edge (a5);
\path[->,orange, bend right, thick] (a5) edge (a6);
\path[->,orange, thick] (a6) edge (a7);
\path[->,orange, bend right, dashed] (a7) edge (a1);
\end{tikzpicture}
\caption{$\Aut\backslash\Stab\D_\infty(K_2)/\CC = \big(\Stab\D_\infty(K_2)/\CC\big)/\ZZ$}\label{fig:K2}
\end{figure}
\begin{proposition}
$\gldim^{-1}(1)\subset \hua{K}_0$ is exactly the area in Figure~\ref{fig:K2}
bounded by $k_0$, $k_1$ and $y$-axis.
Outside this area, $x$-axis gives $1-\gldim$.
Thus we have
\[
    \gldim=\max\{1-x,1\}.
\]
\end{proposition}
\begin{proof}
By the description of $\Aut\backslash\Stab\D^b(\coh \mathbb{P}^1)/\CC$ (see \cite{Ok,Qiu} for details),
this is a direct calculation.
\end{proof}

\section{Further studies: Contractibility of spaces of stability conditions}\label{sec:contractible}
In this section, we explain the potential application of global dimension function
to the contractibility conjecture of spaces of stability conditions.
\subsection{Dynkin case}
Let $Q$ be a Dynkin quiver.
\begin{conjecture}
Regard $\gldim$ as a function on $\Stab\DQ/\CC$.
Then the preimage $\gldim^{-1}[\cd,s)$ contracts to the Gepner point $\sigma_G$
for any real number $s>\cd$.
\end{conjecture}

A much wild guess is the following.
\begin{question}
The function $\gldim$ on $\Stab\DQ/\CC$ is piecewise differentiable
and behaves like a Morse function with a unique critical point $\sigma_G$.
\end{question}

It is also interesting to consider the stability conditions such that every indecomposable is stable.
Such a subspace of $\Stab\DQ$ is closely related to the fundamental domain
of $\Stab\D_{fd}(\qq{2})$ for the Calabi-Yau-2 case (or Kleinian singularities), cf. \cite{B2}.
\begin{question}
It is also interesting to study the space of total stability conditions, or $\operatorname{ToSt}(Q)$
for type $D$ and $E$.
\end{question}
In fact, this question has been just solved in a follow-up work \cite{QZ}.
\subsection{Coherent sheaves on $\bP^n$}
Let $\D^b(\coh \bP^n)$ be the derived category of coherent sheaves on $\bP^n$
for some $n\ge1$.
Recall that for $n=1$, we have \eqref{eq:KQ}.

We have the following conjecture.
\begin{conjecture}
The minimal value of $\gldim$ on $\Stab\D^b(\coh \bP^n)$ is $n$.
\end{conjecture}

This conjecture has been confirmed by Kikuta-Ougchi-Takahashi,
see \cite[Thm.~4.2 and Ex.~2.5]{KOT},
where they show that $\gldim\D^b(\coh \bP^n)\ge n$.
The equality follows immediately as there is a stability condition $\sigma$ with $\gldim=n$,
where the heart of $\sigma$ is the Beilinson heart $\h$ with simples
$$\Sim\h=\hua{O}[n],\hua{O}(1)[n-1]\ldots, \hua{O}(n)$$
and central charges $Z(\hua{O}(i)[n-i])=1$.

Moreover, recall the following definition of geometric stability conditions.

\begin{definition}
A stability condition $\sigma$ on $\D^b(\coh \bP^n)$ is \emph{geometric} if
all skyscraper sheaves $\hh{O}_x$ are $\sigma$-stable with the same phase.
\end{definition}

Note that the conjecture above holds for $n=1$ by the calculation in Section~\ref{sec:P1}.
In the follow-up work \cite{FLLQ}, we prove the following conjecture.
One expects similar result should also hold for $n>2$.

\begin{conjecture}
Regard $\gldim$ as a function on $\Stab\D^b(\coh \bP^2)$.
Then $\gldim^{-1}(2)$ consists of almost all geometric stability conditions and
$\gldim^{-1}[2,s)$ contracts $\gldim^{-1}(2)$, for any real number $s>2$.
\end{conjecture}

\appendix
\section{$q$-stability conditions on Calabi-Yau-$\XX$ categories}\label{sec:q}
Let $\D_{\XX}$ be a triangulated category with a distinguish auto-equivalence
$$\XX \colon \D_{\XX} \to \D_{\XX}.$$
We will write $E[l \XX]$ instead of $\XX^l(E)$ for
$l \in \ZZ $ and $E \in \D_{\XX}$.
Set $$R=\ZZ[q^{\pm 1}]$$ and define the $R$-action on $K(\D_{\XX})$ by
\[
    q^n \cdot [E] := [E[n \XX]].
\]
Then $K(\D_{\XX})$ has an $R$-module structure.
Let $\Aut\D_\XX$ be the group of auto-equivalences of $\D_\XX$
that commute with $\XX$ and
$\Hom^{\ZZ^2}(M,N)\colon=\bigoplus_{k,l\in\ZZ}\Hom(M,N[k+l\ZZ])$.

\subsection{Calabi-Yau-$\XX$ categories from quivers}
For an acyclic quiver $Q$, denote by $\qq{\XX}$
the Calabi-Yau-$\XX$ Ginzburg differential $\ZZ\oplus\ZZ[\XX]$ graded algebra of $Q$,
that is constructed as follows (cf. \cite{IQ1} and \cite[Sec.~7.2]{K}):
\begin{itemize}
\item   Let $\overline{Q}$ be the graded quiver whose vertex set is $Q_0$
and whose arrows are: the arrows in $Q$ with degree $0$;
an arrow $a^*:j\to i$ with degree $2-\XX$ for each arrow $a:i\to j$ in $Q$;
a loop $e^*:i\to i$ with degree $1-\XX$ for each vertex $e$ in $Q$.
\item   The underlying graded algebra of $\qq{\XX}$ is the completion of
the graded path algebra $\k \overline{Q}$ in the category of graded vector spaces
with respect to the ideal generated by the arrows of $\overline{Q}$.
\item   The differential of $\qq{\XX}$ is the unique continuous linear endomorphism homogeneous
of degree $1$ which satisfies the Leibniz rule and
takes the following (non-zero) values on the arrows of $\overline{Q}$
\[
    \diff \sum_{e\in Q_0} e^*=\sum_{a\in Q_1} \, [a,a^*] .
\]
\end{itemize}
Write $\DXQ$ for $\D_{fd}(\mod  \qq{\XX})$,
the \emph{finite dimensional derived category} of $\qq{\XX}$,
which admits the Serre functor $\XX$ that corresponds to grading shift of $(0,1)$.

Note that there is an embedding
\begin{gather}\label{eq:Lag}
    \Lag\colon\DQ\to\DXQ,
\end{gather}
induced from the projection $\qq{\XX}\to \k Q$, which is a Lagrangian immersion,
in the sense that
\begin{gather}\label{eq:HOM}
    \RHom_{\DXQ}(\hh{L}(L),\hh{L}(M))=\RHom_{\DQ}(L,M)
        \oplus D\RHom_{\DQ}(M,L)[-\XX],
\end{gather}
Moreover, $\Lag$ also induces an $R$-isomorphism
\begin{gather}\label{eq:KK}
    \Lag_*\colon K(\DXQ)\cong_R  K(\DQ) \otimes R=R^n,
\end{gather}
where
the simple $\k Q$-modules provide a $\ZZ$-basis for $K(\DQ)\cong\ZZ^n$ and
the simple $\qq{\XX}$-modules provide an $R$-basis for $K(\DXQ)$.

By abuse of notation, we will not distinguish objects in $\DQ$
and their images in $\DXQ$ (under the fixed canonical Lagrangian immersion $\Lag$).
\subsection{$q$-Stability conditions}
We recall $q$-stability conditions from \cite{IQ1}.
\begin{definition}\cite[Def.~3.4]{IQ1}
Suppose that $\D$ is a triangulated category with Grothendieck group $K(\D_\XX)\cong_R R^n$.
An $q$-stability condition
consists of a (Bridgeland) stability condition $\sigma=(Z,\hh{P})$ on $\D_\XX$ and
a complex number $s\in\CC$, satisfying
\begin{gather}\label{eq:X=s}
    \XX(\sigma)=s \cdot \sigma.
\end{gather}
We may write $\sigma[\XX]$ for $\XX(\sigma)$.
Denote by $\XStab_s\D_{\XX}$,
the space of $q$-stability conditions
consisting of $(\sigma, s)$ with $q$-support property.
\end{definition}

We have the following results.

\begin{proposition}\cite[Thm.~3.10]{IQ1}
\label{pp:IQ1}
The projection map defined by taking central charges
\begin{equation}\label{eq:forgetful}
\pi_Z \colon \XStab_s\D_\XX \longrightarrow \Hom_{R}(K(\D_{\XX}),\CC_s),
\quad (Z,\sli) \mapsto Z
\end{equation}
is a local isomorphism of topological spaces.
In particular, $\pi$ induces a complex structure on $\XStab_s\D_{\XX}$.
\end{proposition}

\begin{theorem}\cite[Thm.~5.9]{IQ1}\label{thm:IQ1}
Let $\sigma=(Z,\hua{P})$ be a stability condition on $\Stab\DQ$
and let $s\in\CC$ satisfies
\[
    \Re(s)\ge\gldim\sigma+1.
\]
Then $\Lag$ induces a $q$-stability condition $(\sigma^{\Lag}_s,s)$
such that $\sigma^{\Lag}_s=(Z^{\Lag}_s,\hua{P}^{\Lag}_s)$ is defined as
\begin{itemize}
  \item $Z^{\Lag}_s=q_s\circ \big(  Z \otimes 1 \big)\colon K(\DXQ)\to\CC$
  (recall that we have \eqref{eq:KK}), where
   $q_s$ is the specialization
    \[
        q_s\colon\CC[q,q^{-1}]\to\CC,\quad q\mapsto e^{\mathbf{i} \pi s}.
    \]
  \item $\hua{P}^{\Lag}_s(\phi)=\<   \hua{P}(\phi+k\Re(s)) [k\XX] \mid k\in\ZZ\>.$
\end{itemize}
\end{theorem}

As we have calculate the image of $\gldim$ on $\Stab\DQ$ in Theorem~\ref{thm:Dynkin},
a direct corollary is the following.

\begin{corollary}
Let $Q$ be a Dynkin quiver and $\Re(s)\ge \cd+1$.
Then $\XStab_s\DXQ$ is non-empty.
\end{corollary}

In particular, for any $\Re(s)\ge\cd+1$, denote by $\ssl$
the induced $q$-stability condition in $\XStab_s\DXQ$
from the Gepner point $\sigma_G\in\Stab\DQ$ with $\gldim\sigma_G=\cd$.
In the rest of this section, we shall prove that $\ssl$
is a Gepner point of $\XStab_s\D_\XX$,
in the sense that it satisfies one more Gepner equation \eqref{eq:Phi=s} other than \eqref{eq:X=s}.

\subsection{Center of the braid group}
\begin{definition}\cite{BSai}
Denote by $\Br(Q)$ the braid group (a.k.a Artin group) associated to a Dynkin quiver $Q$,
which admits the following presentation
\[\begin{array}{rll}
    \Br(Q)=\< b_i, i\in Q_0 \mid &\operatorname{Co}(b_i,b_{j}),
        \text{ no arrows between $i$ and $j$}\;;\\
        & \operatorname{Br}(b_j,b_k), \text{ $\exists!$ arrow between $i$ and $j$} \>.
\end{array}\]
where $\operatorname{Co}(a,b)$ means the commutation relation $ab=ba$
and $\Br(a,b)$ means the braid relation $aba=bab$.
\end{definition}
Provided the labeling of vertices as in \eqref{eq:labeling},
define
\begin{gather}
    \zeta_Q=b_n \circ ... \circ b_1,\\
    \delta_Q=\begin{cases}
    1,    \quad &\text{if $Q$ is of type $A_n, D_{2l+1}, E_6$};\\
    1/2,  \quad &\text{if $Q$ is of type $D_{2l}, E_7, E_8$}.
    \end{cases}
\end{gather}
Then it is well-known that $z_Q=\zeta_Q^{\delta_Q h_Q}$
generates of the center of the braid group $\Br(Q)$.

\begin{definition}
An object $S$ in a CY-$\XX$ category $\D$ is \emph{($\XX$-)spherical} if
$$\Hom(S, S[k+l\XX])\begin{cases}
\k,& \text{if $k=0$ and $l\in\{0,1\}$;}\\
0,& \text{otherwise},
\end{cases}$$ and induces
a \emph{twist functor} $\twi_S\in\Aut\D$, such that \[
    \twi_S(X)=\Cone\left(S\otimes\Hom^{\ZZ^2}(S,X)\to X\right)
\]
with inverse
\[
    \twi_S^{-1}(X)=\Cone\left(X\to S\otimes\Hom^{\ZZ^2}(X,S)^\vee \right)[-1].
\].
\end{definition}
In the case when $\DXQ=\D_{fd}(\qq{\XX})$,
any simple $S_i$ for $i\in Q_0$ in the canonical heart
$$\h_Q^\XX\colon=\mod\qq{\XX}$$
is spherical and
the \emph{spherical twist group} is
\[
    \ST_\XX(Q)\colon=\< \twi_{S_i} \mid i\in Q_0 \>\subset\Aut\DXQ.
\]

We have the following.

\begin{theorem}\cite[Thm.~6.6]{IQ1}\label{pp:Br}
There is a canonical isomorphism $\Br(Q)\cong\ST_\XX(Q)$,
sending the standard generators to the standard ones.
\end{theorem}
Thus the generator $\zeta_Q$ of $Z(\Br(Q))$
becomes
\[
    \zql=\twi_{S_n}\circ\cdots\circ\twi_{S_1}.
\]

\subsection{$\XX$-Auslander-Reiten functor}
First, we prove the following lemma.
\begin{lemma}
$\Lag$ induces an injection $\Lag_*\colon\Aut\DQ\to\Aut\DXQ$
that fits into the commutative diagram
\[\xymatrix@R=3pc@C=3pc{
    \DQ \ar[r]^{\Phi} \ar[d]_{\Lag} &\DQ \ar[d]^{\Lag}\\
    \DXQ \ar[r]^{\Lag_*(\Phi)} &\DXQ,
}\]
for any $\Phi\in\Aut\DQ$.
\end{lemma}
\begin{proof}
 By Kozsul duality, $$\DQ=\D_{fd}(\k Q)\cong\per\ee_Q,$$
where
 $$\ee_Q=\RHom_{\DQ}(S_Q,S_Q)$$
is the dg endomorphism algebra for $S_Q=\bigoplus S_i$.
Similarly, we have
$$\DXQ=\D_{fd}(\qq{\XX})\cong\per\ee_Q^\XX,$$
where
$$\ee_Q^\XX=\RHom^{\ZZ^2}_{\DXQ}(S_Q, S_Q)$$
is the differential $\ZZ^2$-graded endomorphism algebra.
Then any auto-equivalence $\Phi$ in $\Aut\DQ$ maps $\ee_Q$
to another dg endomorphism algebra
$$\Phi(\ee_Q)=\RHom^{\ZZ^2}_{\DXQ}(\Phi(S_Q), \Phi(S_Q))$$
that $\Phi$ can be realized as $\per\ee_Q\to\per\Phi(\ee_Q)$.
After applying $\Lag$ that passes to $\DXQ$,
we obtain an auto-equivalence $$\Lag_*(\Phi)\colon\per\ee_Q^\XX\to\per\Lag_*\left(\Phi(\ee_Q^\XX)\right).$$
Finally, if $\Lag_*(\Phi)=\id$ preserves $S_i$ in $\DXQ$ (and the $\Hom$s between them),
then $\Phi$ preserves them in $\DQ$, which must be identity.
\end{proof}

Now, let $\txl=\Lag_*(\AR)\in\Aut\DXQ$.
We have the following.

\begin{proposition}
\label{pp:center}
Let $Q$ be a Dynkin quiver. Then
\begin{gather}
   \txl=[\XX-2]\circ\zql
\end{gather}
satisfies $(\txl)^h=[-2]$.
\end{proposition}
\begin{proof}
The calculation is exactly the same as the Calabi-Yau-$N$ case in the proof of \cite[Prop.~6.4.1]{Qiu}.
Note that the assumption there, i.e. the isomorphism $\Br(Q)\cong\ST_N(Q)$,
has been proved in Theorem~\ref{pp:Br} (cf. \cite{QW}).
\end{proof}

Recall we have a Gepner point $\sigma_G=(Z_G, \hua{P}_G)$ on $\Stab\DQ$
and it induces a $q$-stability condition $(\ssl,s)$
for $\Re(s)\ge \cd+1$, where $\ssl=(Z^{\Lag}_s,\hua{P}^{\Lag}_s)$ is constructed in Theorem~\ref{thm:IQ1}.

\begin{theorem}\label{thm:CYX}
$\ssl\in\XStab_s\DXQ$ satisfies the Gepner equation
\begin{gather}\label{eq:AR=}
    \txl(\sigma)=(-\frac{2}{h}) \cdot \sigma
\end{gather}
for $\Re(s)\ge\cd+1$.
\end{theorem}
\begin{proof}
As $\txl$ is induced from $\AR$ via $\Lag$, we have
$\txl= \AR\otimes R $ on the Grotendieck groups.
Thus, for the central charge we have
\[\begin{array}{rcl}
    Z_{G,s}^\Lag\circ\txl &=& \left( q_s\circ(Z_G\otimes R) \right) \circ \txl\\
    &=&q_s\circ \left( (Z_G\otimes R)\circ \txl \right) \\
    &=&q_s\circ\left(  (Z_G\circ\AR)\otimes R\right) \\
    &=&q_s\circ  \left(  (e^{2\pi\mathbf{i}/h}\cdot Z_G)\otimes R\right) \\
    &=& e^{2\pi\mathbf{i}/h}\cdot \big( q_s\circ  (Z_G\otimes R) \big) \\
    &=& e^{2\pi\mathbf{i}/h}\cdot Z_{G,s}^\Lag,
\end{array}\]
where we use the Gepner property of $\sigma_G$ that $Z_G\circ \AR= e^{2\pi\mathbf{i}/h}\cdot Z_G$.
For the slicing, we have $\txl(\hua{P})=\AR(\hua{P})$
(recall that we identify $\hua{P}(\phi)\subset\DQ$ with its image in $\DXQ$ under $\Lag$)
and hence
\[\begin{array}{rcl}
    \txl\big( \hua{P}_{G,s}^\Lag(\phi) \big)
        &=&  \big\<  \AR\big( \hua{P}(\phi-k\Re(s))  \big) [k\XX] \mid k\in\ZZ \big\>\\
        &=&  \big\<  \hua{P}_{(-2/h)}(\phi-k\Re(s) )  [k\XX] \mid k\in\ZZ \big\>\\
        &=& ( \hua{P}_{G,s}^\Lag )_{(-2/h)} (\phi)
\end{array}\]
where we use the Gepner property of $\sigma_G$ that $\AR(\hua{P})=\hua{P}_{(-2/h)}$.
Thus $\ssl$ satisfies \eqref{eq:AR=}.
\end{proof}


\begin{remark}
A reachable Lagrangian immersion $\Lag'$, by definition,
is of the form $\Lag'=\Upsilon\circ\Lag$ for $\Upsilon\in\Aut\DXQ$,
where $\Lag$ is the fixed initial Lagrangian immersion in \eqref{eq:Lag}.
By the construction in Theorem~\ref{thm:IQ1},
we have
\[
    \sigma_{G,s}^{\Lag'}=\Upsilon\circ\ssl,
\]
which solves the equation
\[
    \AR_\XX^{\Lag'}(\sigma)=(-\frac{2}{h})\cdot\sigma,
\]
for
\[
    \AR_\XX^{\Lag'}=(\Lag')_*(\AR)=
    \Upsilon\circ\txl\circ\Upsilon^{-1}.
\]
Therefore, all such Gepner points $\sigma_{G,s}^{\Lag'}$
correspond to the same point $\sigma_G$ in $\Stab\DXQ/\CC$.
\end{remark}


\end{document}